\documentclass[a4paper,12pt]{amsart}

\usepackage{amsmath, amssymb, amsthm}
\usepackage{hyperref}
\usepackage{xypic}
\usepackage{graphicx}

\usepackage{verbatim}
\usepackage{enumitem}
\usepackage{xcolor}
\allowdisplaybreaks

\newcounter{theorem}
\newtheorem{thm}[theorem]{Theorem}
\newtheorem{lemma}[theorem]{Lemma}
\newtheorem{prop}[theorem]{Proposition}

\newtheorem{defn}[theorem]{Definition}

\newtheorem{thmx}{Theorem}

\theoremstyle{remark}
\newtheorem*{remark*}{Remark}
\newtheorem{remark}[theorem]{Remark}
\newtheorem{example}[theorem]{Example}

\numberwithin{equation}{section}
\numberwithin{theorem}{section}

\newcommand{\e}{\varepsilon}

\newcommand{\N}{\mathbb{N}}
\newcommand{\Z}{\mathbb{Z}}
\newcommand{\supp}{\mathrm{supp}}
\newcommand{\id}{\mathrm{id}}
\newcommand{\homeo}{\mathrm{Homeo}}
\newcommand{\diam}{\mathrm{diam}}
\renewcommand{\setminus}{\backslash}
\renewcommand{\emptyset}{\varnothing}

\begin{document}

\title[Simplicity of Commutator Subgroups of Full Groups]{Comparison and Simplicity of Commutator Subgroups of Full Groups}

\author{Hung-Chang Liao}
\address{Department of Mathematics and Statistics, University of Ottawa,
	\newline 150 Louis-Pasteur Pvt, Ottawa, ON, Canada K1N 6N5}
\email{hliao@uottawa.ca}

\maketitle

\begin{abstract} We show that for a minimal, second countable, locally compact Hausdorff \'etale groupoid whose unit space is homeomorphic to the Cantor set, if the groupoid has comparison then the commutator subgroup of its full group is simple. This generalizes a result of Bezuglyi and Medynets for Cantor minimal systems and complements Matui's results for topological full groups. 
\end{abstract}

\section{Introduction}
\renewcommand{\thethmx}{\Alph{thmx}}

Given a dynamical system, its full group, roughly speaking, consists of automorphisms of the space that respect the orbits. In ergodic theory, the notion of full groups was introduced by Dye  (\cite{Dye:AJM1, Dye:AJM2}). His celebrated reconstruction theorem shows that for ergodic measure-preserving transformations, the full group is a complete invariant for orbit equivalence. In the topological setting, Giordano, Putnam, and Skau (\cite{GPS:Israel}) initiated a systematic study of the full groups (building upon earlier work of Krieger \cite{Krieger:MathAnn} and Renault \cite{Renault:Book}). For a Cantor minimal system, they defined its full group, a direct analogue of the full group for a measure-preserving transformation, and its topological full group, which consists of elements in the full group that are continuous in a suitable sense. They proved that these are complete invariants for orbit equivalence and flip conjugacy, respectively. 

There has been a considerable amount of research on various aspects of these groups, in particular their algebraic properties. In the measurable setting the full groups were shown by Eigen in \cite{Eigen:Israel} (which is inspired by a work of Fathi \cite{Fathi:Israel}). In \cite{Matui:IJM} Matui proved that for the topological full group of a Cantor minimal system, its commutator subgroup is simple (and finitely generated if the system is a minimal subshift). Later, thanks to the remarkable work of Juschenko and Monod on amenability (\cite{JM:Annals}), these commutator subgroups turned out to be the first examples of simple, infinite, finitely generated amenable groups.

The definitions of full groups and topological full groups were generalized by Matui (\cite{Matui:PLMS, Matui:Crelle}) to \'etale groupoids whose unit spaces are homeomorphic to the Cantor set. Note that the assumption on the unit space is natural because full groups and topological full groups might become uninteresting when the space is connected (see the discussion after Proposition \ref{prop:MeasurePerserving}). Many aforementioned results (but not all) were also established at this level of generality. In particular, it is proved in \cite{Matui:Crelle} that for a minimal essentially principal \'etale groupoid that is either almost finite or purely infinite, the commutator subgroup of the topological full group is simple.\footnote{This indeed generalizes the result for Cantor minimal systems because the transformation groupoids arising from these systems are almost finite; see \cite[Lemma 6.3]{Matui:PLMS}.} It is natural to ask whether the same simplicity result holds for the commutator subgroups of full groups of \'etale groupoids. For Cantor minimal systems this is confirmed by Bezuglyi and Medynets (\cite{BM:Colloq}). In this paper we extend their result to a large class of minimal \'etale groupoids, namely those possessing the comparison property. 

\begin{thmx} [Theorem \ref{thm:Simplicity}]
	Let $\mathcal{G}$ be a minimal, second countable, \'etale groupoid whose unit space is a Cantor space. Suppose $\mathcal{G}$ has comparison. Then the commutator subgroup $[\mathcal{G}]'$ of the full group $[\mathcal{G}]$ is simple.
\end{thmx}

The notion of comparison already appeared implicitly in the work of Glasner and Weiss \cite{GlasnerWeiss:IJM}. Very roughly, for a dynamical system we say a set is subequivalent to another set if it can be divided into parts such that the parts can be translated into the other set (this type of relations goes back to Hopf \cite{Hopf:TAMS}). Then we say a dynamical system has comparison if the subequivalence is completely determined by the invariant measures. The formal definition of comparison for topological dynamics was introduced by Winter in 2012 and featured prominently in \cite{Kerr:JEMS, KerrSzabo} for its intimate connections with classification of $C^*$-algebras. In \cite{DZ:arXiv} Downarowicz and Zhang proved that for a large class of countable discrete amenable groups, every action on the Cantor set has comparison. Generalizations of the definition of comparison to \'etale groupoids appeared in \cite{ABBL:arXiv} and \cite{Ma:arXiv}. It is known that a minimal second countable \'etale groupoid whose unit space is homeomorphic to the Cantor set has comparison if it is either almost finite or purely infinite (in the sense of Matui \cite{Matui:PLMS, Matui:Crelle}). \footnote{This goes back to Matui's work (see \cite[Lemma 6.7]{Matui:PLMS} and \cite[Proposition 4.11]{Matui:Crelle}) although comparison was not explicitly defined there; see \cite{ABBL:arXiv} and \cite{Ma:arXiv} for generalizations of these results to groupoids that are not necessarily minimal and whose unit spaces are not necessarily totally disconnected.}

\subsection*{Acknowledgements}
I would like to thank Aaron Tikuisis and Thierry Giordano for many helpful comments. This research is supported by NSERC Discovery Grants and the Fields Institute.

\section{Full groups of \'etale groupoids}

For a topological groupoid $\mathcal{G}$, we let $\mathcal{G}^{(0)}$ denote the unit space of $\mathcal{G}$ and let $r$ and $s$ denote the range and source map, respectively. We say $\mathcal{G}$ is \emph{\'etale} if $r$ and $s$ are local homeomorphisms. For general theory of \'etale groupoids we refer the reader to \cite{Renault:Book},  \cite{Renault:Irish}, and \cite{Sims:Notes}. Throughout this article \textbf{every \'etale groupoid is assumed to be locally compact and Hausdorff.  }

Let $\mathcal{G}$ be an \'etale groupoid. A subset $U$ of $\mathcal{G}$ is called a \emph{bisection} (or \emph{$\mathcal{G}$-set}) if $r\vert_U$ and $s\vert_U$ are injective. Since $r$ and $s$ are local homeomorphisms, the topology of $\mathcal{G}$ has a base consisting of open bisections. To every open bisection $U$ we associate a partial homeomorphism
\[
\sigma_U := r\circ (s\vert_U)^{-1}
\]
from $s(U)$ onto $r(U)$. 

For a point $x$ in $\mathcal{G}^{(0)}$, the \emph{orbit} of $x$ is the set $\{ r(g) : g\in \mathcal{G},\; s(g) = x  \}$, that is, the set of points in $\mathcal{G}^{(0)}$ that are translates of $x$ by elements in $\mathcal{G}$. We say  an \'etale groupoid $\mathcal{G}$ is \emph{minimal} if every orbit is dense in $\mathcal{G}^{(0)}$.

\begin{defn} \label{defn:InvariantMeasure}
	Let $\mathcal{G}$ be an \'etale groupoid. A regular Borel probability measure $\mu$ on $\mathcal{G}^{(0)}$ is  \emph{$\mathcal{G}$-invariant} if 
	\[
	\mu( s(E) ) = \mu(r(E))
	\]
	for every Borel set $E$ that is contained in an open bisection. 
	
	We let $M(\mathcal{G})$ denote the set of all $\mathcal{G}$-invariant regular Borel probability measures.
\end{defn}

\begin{remark}
	If $\mathcal{G}^{(0)}$ is metrizable then any finite Borel measure on $\mathcal{G}^{(0)}$ is regular according to \cite[Theorem 17.10]{Kechris:Book}. This happens, for example, if $\mathcal{G}$ is second countable (see below). 
\end{remark}

Recall that if a locally compact Hausdorff space is second countable, then it is separable and metrizable (in fact Polish; see \cite[Theorem 5.3]{Kechris:Book}). Therefore if an \'etale groupoid is second countable, then its topology has a countable base consisting of open bisections, and its unit space can be equipped with a compatible metric.

\begin{lemma} \label{lem:Kerr} Let $\mathcal{G}$ be a minimal, second countable, \'etale groupoid whose unit space is compact and infinite. Let $d$ be a compatible metric on $\mathcal{G}^{(0)}$.
	\begin{enumerate}
		\item[(1)] For every nonempty open subset $B$ of $\mathcal{G}^{(0)}$, there is a constant $\eta > 0$ such that $\mu( B ) \geq \eta$ for all $\mu\in M(\mathcal{G})$.
		\item[(2)] For every $\e > 0$ there is a constant $\delta > 0$ such that whenever $A$ is a subset of $\mathcal{G}^{(0)}$ satisfying $\diam(A) < \delta$, we have $\mu(A) < \e$ for all $\mu\in M(\mathcal{G})$.
	\end{enumerate} 
\end{lemma}
\begin{proof}
	(1). Since $\mathcal{G}$ is minimal and second countable, we have $\mu(B) > 0$ for every $\mu\in M(\mathcal{G})$. The set $M(\mathcal{G})$, when viewed as a subset of the dual of $C(\mathcal{G}^{(0)})$, is compact in the weak$^*$ topology. Therefore the result follows from \cite[Lemma 3.3]{Kerr:JEMS} (with $A$ in the lemma being the empty set).
	
	(2). Let $\e > 0$ be given. Since $\mathcal{G}$ is minimal and $\mathcal{G}^{(0)}$ is infinite, every orbit is infinite. This implies that $\mu(\{x\}) = 0$ for all $x\in \mathcal{G}^{(0)}$ and $\mu\in M(\mathcal{G})$. By \cite[Lemma 9.1]{Kerr:JEMS} for every $x\in \mathcal{G}^{(0)}$ there exists a constant $\delta_x > 0$ such that the open ball $B(x; \delta_x) := \{y\in \mathcal{G}^{(0)}: d(y,x) < \delta_x   \}$ satisfies 
	\[
	\mu( B(x; \delta_x)   ) < \e
	\] 
	for all $\mu\in M(\mathcal{G})$. Since $\mathcal{G}^{(0)}$ is compact and the open balls $B(x; \delta_x)$ cover $\mathcal{G}^{(0)}$, there exists a constant $\delta > 0$ (i.e., a Lebesgue number) such that every subset of $\mathcal{G}^{(0)}$ having diameter less than $\delta$ is contained in some $B(x; \delta_x)$. This concludes the proof.
\end{proof}

We now recall the definition of full groups for \'etale groupoids. We also recall the definition of topological full groups since some lemmas in Section \ref{sec:comparison} actually produce elements in these subgroups.

\begin{defn}  [{\cite[Definition 2.3]{Matui:PLMS}}] \label{defn:FullGroup}
	Let $\mathcal{G}$ be an \'etale groupoid. The \emph{full group} $[\mathcal{G}]$ consists of $\alpha\in \homeo(\mathcal{G}^{(0)})$ such that for every $x\in \mathcal{G}^{(0)}$ there is an element $g\in \mathcal{G}$ satisfying $x = s(g)$ and $\alpha(x) = r(g)$.
	
	The \emph{topological full group} $[[\mathcal{G}]]$ is the subgroup of $[\mathcal{G}]$ consisting of elements $\alpha$ for which there is a compact open bisection $U$ such that $\alpha = \sigma_U$ (in particular $s(U) = r(U) = \mathcal{G}^{(0)}$). 
\end{defn}

As noted in \cite{Matui:PLMS}, for the transformation groupoid $\mathcal{G}_\varphi$ arising from a Cantor minimal system $(X, \varphi)$ (i.e., $\varphi$ is a homeomorphism on the Cantor set $X$), the definitions agree with the ones given in \cite{GPS:Israel}. Given a homeomorphism $\alpha$ on a topological space $X$ and a Borel measure $\mu$ on $X$,  define the measure $\alpha^*\mu$ on $X$ by setting
\[
(\alpha^*\mu)(E) = \mu(\alpha(E))
\]
for every Borel subset $E$ of $X$. For a Cantor minimal system, it is readily seen that any homeomorphism in the full group acts trivially on the set of invariant Borel probability measures. The next proposition extends this fact to any second countable \'etale groupoid.

\begin{prop} \label{prop:MeasurePerserving}
	Let $\mathcal{G}$ be a second countable \'etale groupoid, and let $\alpha \in [\mathcal{G}]$. Then $\alpha^*\mu = \mu$ for every $\mu\in M(\mathcal{G})$.
\end{prop}
\begin{proof}
	Let $\{U_1, U_2,...\}$ be a countable base for the topology of $\mathcal{G}$ consisting of open bisections. For each $n\in \N$ define
	\[
	E_n := \{ g\in U_n: \alpha(s(g)) = r(g)  \}.
	\]
	Then $s(E_n)$ is closed in $s(U_n)$ because 
	\[
	s(E_n) = \{ x\in s(U_n) : \sigma_{U_n}(x) = \alpha(x)  \}.
	\]
	Since $\alpha$ belongs to the full group $[\mathcal{G}]$, the collection $\{s(E_n)\}_{n=1}^\infty$ covers the unit space $\mathcal{G}^{(0)}$. We define
	\[
	B_n := s(E_n) \setminus \left( \bigcup_{i=1}^{n-1} s(E_i)   \right)     \qquad (n=1,2,...).
	\]
	Then $\{B_1, B_2,...\}$ forms a partition of $\mathcal{G}^{(0)}$ consisting of Borel subsets. Let us set
	\[
	F_n := (s\vert_{U_n})^{-1}(B_n) \qquad \text{ and } \qquad  C_n := r(F_n)
	\]
	for each $n\in \N$. By definition $\sigma_{U_n}$ and $\alpha$ agree on $s(E_n)$, so they also agree on $B_n$. It follows that $C_n = \alpha(B_n)$ and hence $\{C_1, C_2,...\}$ is also a Borel partition of $\mathcal{G}^{(0)}$.  Then for every Borel set $A\subseteq \mathcal{G}^{(0)}$ and $\mu\in M(\mathcal{G})$,
	\begin{align*}
	\mu(A) &= \mu\left(  \bigsqcup_{n=1}^\infty[ A\cap B_n ]   \right) = \sum_{n=1}^\infty \mu(A\cap B_n  ) \\
	&= \sum_{n=1}^\infty \mu(  \sigma_{U_n}(A\cap B_n)     ) = \sum_{n=1}^\infty \mu(  \alpha(A\cap B_n)  ) \\
	&= \sum_{n=1}^\infty \mu( \alpha(A)\cap C_n   ) = \mu\left(  \bigsqcup_{n=1}^\infty [\alpha(A)\cap C_n]  \right) \\
	&= \mu( \alpha(A)  ).
	\end{align*}
	Therefore $\alpha^*\mu = \mu$. 
\end{proof}

Looking at Definition \ref{defn:FullGroup}, one quickly realizes that if the unit space is connected then the topological full group does not carry much information about the dynamics. Although less obvious, already for Cantor minimal systems the same goes for the full group (see \cite[Proposition 1.3]{GPS:Israel}). Therefore we will mostly be working with groupoids whose unit spaces are totally disconnected.\footnote{Here by totally disconnected we mean that the only connected subspaces are one-point sets. A locally compact Hausdorff space is totally disconnected if and only if its topology has a base consisting of clopen sets (see, for example, \cite[II.4, Proposition D]{HW:Book}). An \'etale groupoid whose unit space is totally connected is often called \emph{ample}.}  In this case the topology of the groupoid has a base consisting of compact open bisections (\cite[Proposition 4.1]{Exel:PAMS}). 

Before closing the section, we recall the definition of the support of a homeomorphism, which plays an important role in studying algebraic properties of full groups and topological full groups.

\begin{defn}
	Let $X$ be a topological space and $\alpha\in \homeo(X)$. The \emph{support} of $\alpha$ is defined by
	\[
	\supp(\alpha) :=  \overline{ \{  x\in X: \alpha(x) \neq x  \}},
	\]
	where $\overline{A}$ denotes the closure of $A$.
\end{defn}

The following simple observation will be used many times.

\begin{lemma}
	Let $\alpha$ and $\beta$ be two homeomorphisms on a topological space $X$. Then 
	\[
	\supp(\beta\alpha\beta^{-1})    = \beta(  \supp(\alpha)).
	\]
\end{lemma}
\begin{proof}
	Let $A = \{ x\in X: \alpha(x) \neq x  \}$. Then $\beta \alpha \beta^{-1}(x)\neq x$ if and only if $x\in \beta(A)$. Therefore 
	\[
	\supp( \beta \alpha \beta^{-1} ) = \overline{\beta(A)} = \beta(\overline A) = \beta( \supp(\alpha) ).
	\]
\end{proof}

\section{Comparison and full groups} \label{sec:comparison}

We call a topological space a \emph{Cantor space} if it is homeomorphic to the Cantor set. As mentioned in the introduction, the formal definition of comparison for \'etale groupoids appeared in \cite{ABBL:arXiv} and \cite{Ma:arXiv}.\footnote{In \cite{ABBL:arXiv} the groupoids were assumed to have totally disconnected unit spaces.} The following definition is \cite[Definition 2.1]{ABBL:arXiv} specialized to minimal groupoids whose unit spaces are Cantor spaces (see \cite[Remark 2.2]{ABBL:arXiv}).

\begin{defn}  \label{defn:Comparison}
	Let $\mathcal{G}$ be a minimal \'etale groupoid whose unit space is a Cantor space. We say $\mathcal{G}$ has \emph{comparison} if for any clopen subsets $A$, $B$ of $\mathcal{G}^{(0)}$ such that $B\neq \emptyset$ and $\mu(A) < \mu(B)$ for all $\mu\in M(\mathcal{G})$ (with $M(\mathcal{G})$ possibly being empty), there is a compact open bisection $U$ such that $s(U) = A$ and $r(U) \subseteq B$. 
\end{defn}

When $M(\mathcal{G})=\emptyset$, comparison simply means that for any clopen subsets $A$, $B$ of $\mathcal{G}^{(0)}$ with $B\neq \emptyset$ there is a compact open bisection $U$ such that $s(U) = A$ and $r(U) \subseteq B$.

\begin{example}
	Let $\mathcal{G}$ be a minimal second countable \'etale groupoid whose unit space is a Cantor space. 
	\begin{enumerate}
		\item[(1)] If $\mathcal{G}$ is purely infinite in the sense of \cite[Definition 4.9]{Matui:Crelle}, then $\mathcal{G}$ has comparison by \cite[Proposition 4.11]{Matui:Crelle}.\footnote{In fact pure infiniteness is equivalent to comparison when $M(\mathcal{G}) = \emptyset$; see \cite[Theorem 5.1]{Ma:arXiv}.}  More concrete examples include \'etale groupoids arising from shifts of finite-type, boundary actions of free groups, and $n$-filling actions in the sense of \cite{JR:JFA}. We refer the reader to \cite{Matui:Crelle} for more examples and discussions. (See also \cite{Ma:arXiv} for generalization to groupoids whose unit spaces are not necessarily totally disconnected.) 
		\item[(2)] If $\mathcal{G}$ is almost finite in the sense of \cite[Definition 6.2]{Matui:PLMS}, then it follows from \cite[Lemma 6.7]{Matui:PLMS} that  $\mathcal{G}$ has comparison (see also \cite[Lemma 3.14]{ABBL:arXiv}). In \cite{DZ:arXiv} it was shown that if $G$ is finitely generated and has locally subexponential growth (for example $\Z^n$) then every free action of $G$ on a Cantor space is almost finite.\footnote{It was actually proved directly that the action has comparison.} For more on almost finite groupoids and almost finite actions, we refer the reader to \cite{ABBL:arXiv, Kerr:JEMS, KerrSzabo, Matui:PLMS, Suzuki:IMRN}.
	\end{enumerate}	
\end{example}

In what follows we present some results on comparing clopen sets by elements in (topological) full groups. Most of the arguments are already contained in \cite{GlasnerWeiss:IJM}, \cite{BM:Colloq}, and \cite{Matui:Crelle}. However, since we have different assumptions and need a few refinements for the next section, we include full details.  

\begin{lemma} \label{lem:ComparisonByFullGroup}
	Let $\mathcal{G}$ be a minimal \'etale groupoid whose unit space is a Cantor space. Suppose $\mathcal{G}$ has comparison. Then for any clopen subsets $A$, $B$ of $\mathcal{G}^{(0)}$ such that $A\neq \mathcal{G}^{(0)}$, $B\neq \emptyset$, and $\mu(A) < \mu(B)$ for all $\mu\in M(\mathcal{G})$ (with $M(\mathcal{G})$ possibly being empty), there is a homeomorphism $\alpha\in [[\mathcal{G}]]$ such that $\alpha(A)\subseteq B$. Moreover, 
	\begin{enumerate}
		\item[(1)] if $B\setminus A$ is nonempty then we may arrange that $\alpha^2 = \id$, and $\supp(\alpha) \subseteq A\cup \alpha(A)$;
		\item[(2)] if  $B\subseteq A$ (which can only happen when $M(\mathcal{G}) = \emptyset$) then we may arrange that $A\cup \supp(\alpha) \neq \mathcal{G}^{(0)}$.
	\end{enumerate}
	
\end{lemma}
\begin{proof}
	We first assume that $B' := B\setminus A$ is nonempty. We also assume that $A' := A\setminus B$ is nonempty, otherwise we can simply take the identity map. Since $\mathcal{G}$ has comparison, regardless of whether $M(\mathcal{G})$ is empty or not there is a compact open bisection $U$ such that $s(U) = A'$ and $r(U) \subseteq B'$. We define the map $\alpha:X\to X$ by 
	\[
	\alpha(x) := \begin{cases} \sigma_U(x) &  \text{ if } x \in A'; \\
	\sigma_U^{-1}(x) & \text{ if } x\in r(U); \\
	x & \text{ otherwise}.  
	\end{cases}
	\]  
	Then $\alpha$ belongs to $[[\mathcal{G}]]$, $\alpha^2 = \id$, and $\supp(\alpha)\subseteq A\cup \alpha(A)$. 
	
	Now suppose  $B$ is contained in $A$ (which implies that $M(\mathcal{G}) = \emptyset$). Since $A\neq \mathcal{G}^{(0)}$, by the previous paragraph we can find $\alpha_1, \alpha_2 \in [[\mathcal{G}]]$ such that $\alpha_1( \mathcal{G}^{(0)} \setminus A  ) \subseteq B$ and $\alpha_2(A) \subseteq \mathcal{G}^{(0)}\setminus A$.  Then the composition $\alpha := \alpha_1\alpha_2$ maps $A$ into $B$. To ensure that $A \cup \supp(\alpha) \neq \mathcal{G}^{(0)}$ we can choose beforehand a nonempty clopen subset $C$ of $\mathcal{G}^{(0)}\setminus A$ such that $C \neq \mathcal{G}^{(0)}\setminus A$, and apply the same construction within the complement of $C$. This ensures that $C$ is disjoint from the support of $\alpha$. 
\end{proof}

\begin{lemma} \label{lem:ComparisonByCommutator}
	Let $\mathcal{G}$ be a minimal, second countable, \'etale groupoid whose unit space is a Cantor space. Suppose $\mathcal{G}$ has comparison. Then for any clopen subsets $A$, $B$ of $\mathcal{G}^{(0)}$ such that $A\neq \mathcal{G}^{(0)}$, $B\neq \emptyset$, and $3\mu(A) < \mu(B)$ for all $\mu\in M(\mathcal{G})$ (with $M(\mathcal{G})$ possibly being empty), there is a homeomorphism $\gamma\in [[\mathcal{G}]]'$ such that $\gamma(A)\subseteq B$. Moreover,
	\begin{enumerate}
		\item[(1)] if $B\setminus A$ is nonempty then we may arrange that $\gamma^2(A)\subseteq B$ and  $\supp(\gamma) \subseteq A\cup \gamma(A)\cup \gamma^2(A)$;
		\item[(2)] if  $B\subseteq A$ (which can only happen when $M(\mathcal{G}) = \emptyset$) then we may arrange that $A\cup \supp(\gamma) \neq \mathcal{G}^{(0)}$.
	\end{enumerate}
\end{lemma}
\begin{proof}
	First assume that $B\setminus A$ is nonempty. By requiring the homeomorphism to act trivially on $A\cap B$, we may assume that $A\cap B = \emptyset$ provided we weaken the hypothesis $3\mu(A) < \mu(B)$ to $2\mu(A) < \mu(B)$ for all $\mu\in M(\mathcal{G})$ (with $M(\mathcal{G})$ possibly being empty). By Lemma \ref{lem:ComparisonByFullGroup}(1) there is an involution $\alpha\in [[\mathcal{G}]]$ such that $\alpha(A) \subseteq B$ and $\supp(\alpha) \subseteq A\cup \alpha(A)$. Note that we may assume $B\setminus \alpha(A)$ is nonempty (this is automatic if $M(\mathcal{G})\neq \emptyset$; if $M(\mathcal{G})$ is empty we can choose a proper nonempty clopen subset $B_0$ of $B$ and apply Lemma \ref{lem:ComparisonByFullGroup}(1) to $A$ and $B_0$). Using Lemma \ref{lem:ComparisonByFullGroup}(1) again we obtain an involution $\beta\in [[\mathcal{G}]]$ such that $\beta(A) \subseteq B\setminus \alpha(A)$ and $\supp(\beta) \subseteq A\cup \beta(A)$. We set 
	\[
	\gamma := [\alpha, \beta].
	\] 
	It is straightforward to check that  
	\[
	\gamma = \alpha\beta \alpha^{-1}\beta^{-1} = \beta\alpha
	\]
	(in effect $\gamma$ cyclically permutes the sets $A$, $\alpha(A)$, and $\beta(A)$). Then $\gamma\in [[\mathcal{G}]]'$ satisfies all the requirements. 
	
	Now suppose $B$ is contained in $A$. By Lemma \ref{lem:ComparisonByFullGroup}(2) there exists a homeomorphism $\alpha \in [[\mathcal{G}]]$ such that $\alpha(A)\subseteq B$ and $A\cup \supp(\alpha) \neq \mathcal{G}^{(0)}$. Let $A'$ be a clopen subset such that $A\cup \supp(\alpha) \subseteq A' \neq \mathcal{G}^{(0)}$. Using Lemma \ref{lem:ComparisonByFullGroup}(2) again we can find a homeomorphism $\beta\in [[\mathcal{G}]]$ such that $\beta(A') \subseteq \mathcal{G}^{(0)}\setminus A'$ and $A'\cup \supp(\beta) \neq \mathcal{G}^{(0)}$. Since $\supp(\beta\alpha^{-1}\beta^{-1}) = \beta (\supp(\alpha))$ and the latter is disjoint from $A$, we see that 
	\[
	\alpha\beta\alpha^{-1}\beta^{-1}(A) = \alpha(A)\subseteq B.
	\]
	Therefore $\gamma := [\alpha, \beta]\in [[\mathcal{G}]]'$ maps $A$ into $B$. Moreover,
	\[
	A\cup \supp(\gamma )  \subseteq A\cup\supp(\alpha)\cup \supp(\beta) \subseteq  A'\cup \supp(\beta) \neq \mathcal{G}^{(0)}.
	\]
\end{proof}

\begin{lemma} [{cf.\ \cite[Proposition 2.6]{GlasnerWeiss:IJM}}]\label{lem:GlasnerWeissIntertwining}
	Let $\mathcal{G}$ be a minimal, second countable, \'etale groupoid whose unit space is a Cantor space. Suppose $\mathcal{G}$ has comparison. Then for any clopen sets $A, B$ of $\mathcal{G}^{(0)}$ such that $A\setminus B \neq \emptyset$, $B\setminus A \neq \emptyset$, and $\mu(A) = \mu(B)$ for all $\mu\in M(\mathcal{G})$ (with $M(\mathcal{G})$ possibly being empty), there is a homeomorphism $\alpha\in [\mathcal{G}]$ such that $\alpha(A) = B$,  $\alpha^2 = \id$, and $\supp(\alpha) \subseteq A\cup B$.  
\end{lemma}
\begin{proof}
	We first consider the case when $M(\mathcal{G})\neq \emptyset$, as the other case follows from a similar (but easier) argument. By setting $\alpha$ to be the identity on $A\cap B$, we may assume that $A$ and $B$ are disjoint. Since $\mathcal{G}$ is minimal there is an element $g\in \mathcal{G}$ such that $x_0 := s(g) \in A$ and $y_0 := r(g) \in B$. Let $d$ be a compatible metric on $\mathcal{G}^{(0)}$, and let $A_1\subseteq A$ be a clopen neighborhood of $x_0$ such that $\diam(A_1) < \frac{1}{2}$ and $\mu(A_1) < \frac{1}{2}\mu(A)$ for all $\mu\in M(\mathcal{G})$ (see Lemma \ref{lem:Kerr}). Set $A_1' := A\setminus A_1$. Using Lemma \ref{lem:ComparisonByFullGroup}(1) we can find an involution $\alpha_1\in [[\mathcal{G}]]$ such that $\alpha_1(A_1') \subseteq B$ and $\supp(\alpha_1)\subseteq A_1'\cup \alpha_1(A_1')$. Note that by choosing a clopen neighborhood of $y_0$ with sufficiently small measures, we may arrange that $\alpha_1(A\setminus A_1)$ does not contain $y_0$. Set $B_1' := \alpha_1(A_1')$ and
	\[
	B_1 := B \setminus B_1'.
	\]
	Then $\mu(B_1) = \mu(A_1)$ for all $\mu\in M(\mathcal{G})$. Now we apply the preceding argument to $B_1$ and $A_1$. Let $B_2\subseteq B_1$ be a clopen neighborhood of $y_0$ such that $\diam(B_2) < \frac{1}{4}$ and  $\mu(B_2) < \frac{1}{2}\mu(B_1)$ for all $\mu\in M(\mathcal{G})$. We set $B_2' := B_1\setminus B_2$. Then by Lemma \ref{lem:ComparisonByFullGroup}(1) there is an involution $\beta_2\in [[\mathcal{G}]]$ such that $\beta_2(B_2'  ) \subseteq A_2$ and $\supp(\beta_2) \subseteq B_2'\cup \beta_2(B_2')$. As before we may assume that $A_2' := \beta_2(B_2')$ does not contain the point $x_0$. For convenience we set $\alpha_2 := \beta_2^{-1}$. 
	
	Let $A_0 = A$ and $B_0 = B$. By repeating the argument in the previous paragraph we obtain two decreasing sequences of clopen neighborhoods $(A_n)$ and $(B_n)$ of $x_0$ and $y_0$, respectively, and a sequence of involutions $(\alpha_n)$ in $[[\mathcal{G}]]$ such that 
	\begin{enumerate}
		\item $\diam(A_n) < \frac{1}{2^{n-1}}$, $\diam(B_n) < \frac{1}{2^{n-1}}$,
		\item $\alpha_n( A_{n-1}\setminus A_n  ) = B_{n-1}\setminus B_n$, and
		\item $\supp(\alpha_n) \subseteq  (A_{n-1}\setminus A_n)\cup (B_{n-1}\setminus B_n)$
	\end{enumerate}
	for all $n\in \N$. Define a map $\alpha:X\to X$ by
	\[
	\alpha(x) :=
	\begin{cases}
	y_0 & \text{ if } x = x_0; \\
	\alpha_n(x) & \text{ if } x \in A_{n-1}\setminus A_n; \\
	\alpha_n^{-1}(x) & \text{ if } x \in B_{n-1}\setminus B_n; \\
	x & \text{ if } x \in \mathcal{G}^{(0)}\setminus (A\cup B).
	\end{cases}	
	\]
	From the condition on the diameters of $A_n$ and $B_n$ we deduce that $\{ x_0 \} = \bigcap_{n} A_n$ and $\{y_0\} = \bigcap_n B_n $, so the map $\alpha$ is indeed defined on all of $X$. It is straightforward to see that $\alpha$ is a homeomorphism in $[\mathcal{G}]$ that satisfies all the required properties. This completes the proof when $M(\mathcal{G}) \neq \emptyset$.
	
	When $M(\mathcal{G}) = \emptyset$ the proof works almost verbatim, except that there is no need to consider measures. Lemma \ref{lem:ComparisonByFullGroup}(1) applies since the relevant clopen sets are disjoint and nonempty.  
\end{proof}

The following proposition goes back to Bezuglyi and Medynets (in the setting of Cantor minimal systems) and plays a fundamental role in establishing simplicity of the commutator subgroups. Roughly speaking, it says that every homeomorphism in the full group can be decomposed into a product of full groups elements whose supports are ``small''.

\begin{prop} [{cf.\ \cite[Lemma 3.2]{BM:Colloq}}]   \label{prop:DecompSmall}
	Let $\mathcal{G}$ be a minimal, second countable, \'etale groupoid whose unit space is a Cantor space. Suppose $\mathcal{G}$ has comparison. Then for any  $\alpha\in [\mathcal{G}]$ and $\e > 0$, there exist $\alpha_1, \alpha_2, ..., \alpha_n$ in $[\mathcal{G}]$ and clopen sets $C_1, C_2,...,C_n$ such that
	\begin{enumerate}
		\item $\alpha = \alpha_1\alpha_2\cdots \alpha_n$,
		\item $\supp(\alpha_i) \subseteq C_i \neq \mathcal{G}^{(0)}$, and
		\item $\mu(C_i) < \e$ for all $\mu\in M(\mathcal{G})$ (which is vacuously satisfied if $M(\mathcal{G}) = \emptyset$)
	\end{enumerate} 
	for every $i\in \{1,...,n\}$. 
\end{prop}
\begin{proof}
	Let $\alpha\in [\mathcal{G}]$ and $\e > 0$ be given. We assume that that $\alpha$ is nontrivial, otherwise there is nothing to prove. We first deal with the case $M(\mathcal{G}) =\emptyset$ (which is easier). Let $A$ be a clopen subset of $\mathcal{G}^{(0)}$ such that $A\cap \alpha(A) = \emptyset$ and $A\cup \alpha(A) \neq \mathcal{G}^{(0)}$. Define $\alpha_1$ by
	\[
	\alpha_1(x) := \begin{cases}
	\alpha(x) & \text{ if } x\in A, \\
	\alpha^{-1}(x) & \text{ if  } x\in \alpha(A), \\
	x & \text{ otherwise}.
	\end{cases}	
	\]
	Then $\alpha_1$ is an involution in $[\mathcal{G}]$ such that $\supp(\alpha_1) = A\cup \alpha(A) \neq \mathcal{G}^{(0)}$. Setting $\alpha_2 := \alpha_1^{-1}\alpha$, we see that $\supp(\alpha_2) \subseteq \mathcal{G}^{(0)}\setminus A \neq \mathcal{G}^{(0)}$ and $\alpha = \alpha_1\alpha_2$. 
	
	Now suppose $M(\mathcal{G})$ is nonempty. Using Lemma \ref{lem:Kerr}(2) we can find a clopen partition $\mathcal{G}^{(0)} = \bigsqcup_{i=1}^n A_i$ such that $\mu(A_i) < \frac{\e}{2}$ for all $i\in \{1,...,n\}$ and $\mu\in M(\mathcal{G})$. For convenience we set $A_1' := A_1\setminus \alpha(A_1)$ and $B_1' := \alpha(A_1)\setminus A_1$ (with $A_1'$ and $B_1'$ possibly being empty). By Proposition \ref{prop:MeasurePerserving} 
	\begin{align*}
	\mu(A_1') &= \mu(A_1) - \mu(A_1\cap \alpha(A_1)  )  \\ 
	&= \mu(\alpha(A_1)) - \mu(A_1\cap \alpha(A_1)  ) = \mu(B_1')
	\end{align*} 
	for all $\mu\in M(\mathcal{G})$.  Therefore by Lemma \ref{lem:GlasnerWeissIntertwining} there is a homeomorphism $\beta\in [\mathcal{G}]$ such that $\beta(B_1') = A_1'$ and $\supp(\beta)\subseteq A_1'\cup B_1'$.
	Define the map $\alpha_1$ on $\mathcal{G}^{(0)}$ by
	\[
	\alpha_1(x) := \begin{cases}
	\alpha(x) & \text{ if } x\in A_1, \\
	\beta(x) & \text{ if } x\in B_1', \\
	x & \text{ otherwise}.
	\end{cases}
	\] 
	Since both $\alpha$ and $\beta$ are in $[\mathcal{G}]$, so is $\alpha_1$. Note that the support of $\alpha_1$ is contained in $A_1\cup B_1'$ ($= A_1\cup \alpha(A_1)$) and that $\supp(\alpha_1^{-1}\alpha) \subseteq \mathcal{G}^{(0)}\setminus A_1$. 
	
	For $i\in \{2,...,n\}$, we apply the same argument inductively to the element $\alpha_{i-1}^{-1}\cdots \alpha_1^{-1}\alpha$ in $[\mathcal{G}]$ and the clopen set $A_i$. In this way we obtain $\alpha_2,...,\alpha_n\in [\mathcal{G}]$ such that for each $i\in \{2,...,n\}$
	\begin{enumerate}
		\item[(1)] $\supp(\alpha_i) \subseteq A_i\cup \alpha_{i-1}^{-1}\cdots \alpha_1^{-1}\alpha(A_i)$, and
		\item[(2)] $\supp( \alpha_i^{-1}\alpha_{i-1}^{-1}\cdots \alpha_1^{-1}\alpha   ) \subseteq \mathcal{G}^{(0)}\setminus A_i$.
	\end{enumerate}
	In fact, we can say more about the support of $\alpha_i^{-1}\alpha_{i-1}^{-1}\cdots \alpha_1^{-1}\alpha$. Since $\supp(\alpha_1^{-1}\alpha)\subseteq \mathcal{G}^{(0)}\setminus A_1$ and $A_2$ is disjoint from $A_1$, the set $\alpha_1^{-1}\alpha(A_2)$ is also disjoint from $A_1$. It follows that $\alpha_2$ is supported outside $A_1$ and hence $\alpha_2^{-1}\alpha_1^{-1}\alpha$ acts as the identity on $A_1\cup A_2$ (as opposed to just $A_2$). By induction, we have
	\begin{enumerate}
		\item[(2')] $\supp( \alpha_i^{-1}\alpha_{i-1}^{-1}\cdots \alpha_1^{-1}\alpha  ) \subseteq \mathcal{G}^{(0)} \setminus ( A_1\cup \cdots \cup A_i  )$
	\end{enumerate}
	for all $i\in \{2,...,n\}$. From (1) we see that $\supp(\alpha_i)$ is contained in the clopen set $C_i :=  A_i\cup \alpha_{i-1}^{-1}\cdots \alpha_1^{-1}\alpha(A_i)$, and $\mu(C_i) < \e$ for all $i\in \{1,...,n\}$ and $\mu\in M(\mathcal{G})$. From $(2')$ we have
	\[
	\supp(\alpha_n^{-1}\cdots \alpha_1^{-1}\alpha) \subseteq \mathcal{G}^{(0)}\setminus (A_1\cup \cdots \cup A_n) = \emptyset. 
	\]
	Therefore $\alpha_n^{-1}\cdots \alpha_1^{-1}\alpha$ is the identity map and $\alpha = \alpha_1\alpha_2\cdots \alpha_n$.
\end{proof}

\section{Simplicity of commutator subgroups of full groups}

In this section we establish the main result (Theorem \ref{thm:Simplicity}). The strategy is the same as \cite{Matui:Crelle} for topological full groups (and in fact, many constructions in this section come directly from \cite{Matui:Crelle}). We first show that any subgroup of $[\mathcal{G}]$ normalized by $[\mathcal{G}]'$ is in fact normal in $[\mathcal{G}]$ (Proposition \ref{prop:Normality}), and then prove that any normal subgroup of $[\mathcal{G}]$ must contain the commutator subgroup $[\mathcal{G}]'$ (Proposition \ref{prop:containment}).

The following lemma is inspired by \cite[Lemma 4.14]{Matui:Crelle}. The novelty here is to treat the finite case ($M(\mathcal{G}) \neq \emptyset$) and infinite case $(M(\mathcal{G}) =\emptyset)$  simultaneously. For this we need to modify the argument given in  \cite{Matui:Crelle}.

\begin{lemma} \label{lem:DecompNormal}
	Let $\mathcal{G}$ be a minimal, second countable, \'etale groupoid whose unit space is a Cantor space. Suppose $\mathcal{G}$ has comparison. If $N$ is a subgroup of $[\mathcal{G}]$ normalized by $[\mathcal{G}]'$, then for any $\tau \in N$ there exist $\tau _1, \tau_2\in N$ such that 
	\begin{enumerate}
		\item $\tau = \tau_1\tau_2$, and
		\item $\supp(\tau_i) \neq \mathcal{G}^{(0)}$ for $i=1, 2$.
	\end{enumerate} 
\end{lemma}
\begin{proof}
	Let $A$ be a nonempty clopen subset of $\mathcal{G}^{(0)}$ such that $A\cap \tau(A) = \emptyset$ and $A\cup \tau(A) \neq \mathcal{G}^{(0)}$. By Lemma \ref{lem:Kerr}(2), if $M(\mathcal{G})$ is nonempty, we may also arrange that $\mu(A) < \frac{1}{16}$ for all $\mu\in M(\mathcal{G})$. Applying Lemma \ref{lem:ComparisonByFullGroup} we can find an element $\sigma_0 \in [[\mathcal{G}]]$ ($\subseteq [\mathcal{G}]$) such that $B := \sigma_0( \tau(A) ) \subseteq \mathcal{G}^{(0)}\setminus (A\cup \tau(A))$. Note that by shrinking $A$ we may assume that $A\cup \tau(A)\cup \tau^{-1}(A)\cup B \neq \mathcal{G}^{(0)}$. Let $A_0\subseteq A$ be a proper nonempty clopen subset of $A$ and set $B_0 := \sigma_0(\tau(A_0))\subseteq B$. Define $\sigma_1, \sigma_2\in [\mathcal{G}]$ by 
	\[
	\sigma_1(x) := \begin{cases}
	\tau(x) & \text{ if } x\in A_0,\\
	\tau^{-1}(x) & \text{ if } x\in \tau(A_0), \\
	x & \text{ otherwise},
	\end{cases}	
	\quad 
	\sigma_2(x) := \begin{cases}
	\sigma_0(x) & \text{ if } x\in \tau(A_0),\\
	\sigma_0^{-1}(x) & \text{ if } x\in B_0,\\
	x & \text{ otherwise},
	\end{cases}
	\]
	Then, as in the proof of Lemma \ref{lem:ComparisonByCommutator}, $\sigma := [\sigma_2, \sigma_1] = \sigma_1\sigma_2\sigma_1^{-1}\sigma_2^{-1} = \sigma_1\sigma_2$ cyclically permutes the sets $A_0$, $\tau(A_0)$, and $B_0$. In particular $\supp(\sigma)\subseteq A_0\cup \tau(A_0)\cup B_0$ and  $\sigma(x) = \tau(x)$ for all $x\in A_0$. 
	
	By Lemma \ref{lem:ComparisonByCommutator}(1) we can find an element $\gamma\in [[\mathcal{G}]]'$ ($\subseteq [\mathcal{G}]'$) such that $\gamma( \tau(A) \cup B  ) \subseteq \mathcal{G}^{(0)}\setminus (A\cup \tau(A)\cup \tau^{-1}(A)\cup B) =: C$ and $\supp(\gamma) \subseteq \tau(A)\cup B\cup C$. Note that in particular $\gamma$ acts as the identity on $A$. Observe that
	\begin{align*}
	\supp( \tau \gamma \sigma^{-1} \gamma^{-1} \tau^{-1}  ) &= \tau\gamma( \supp(\sigma)) \subseteq \tau\gamma(  A_0\cup \tau(A_0)\cup  B_0 )\\
	&\subseteq \tau(A_0 \cup C ) \subseteq  \tau(A) \cup \tau(C).
	\end{align*}
	Since $C$ and $\tau^{-1}(A)$ are disjoint, $\tau(C)$ is disjoint from $A$. It follows that $\tau \gamma \sigma^{-1} \gamma^{-1} \tau^{-1} $ also acts as the identity on $A$. Define $\tau_0, \tau_1$ by
	\[
	\tau_0 := [ \gamma \sigma \gamma^{-1}, \tau     ] \quad  \text{and} \quad \tau_1 := \gamma^{-1}\tau_0 \gamma.
	\]
	Since $N$ is normalized by $[\mathcal{G}]'$, the elements $\tau_0$ and $\tau_1$ belong to $N$ (remember that both $\sigma$ and $\gamma$ are in $[\mathcal{G}]'$). We have 	
	\begin{align*}
	\supp(\tau_0) &= \supp( \gamma \sigma \gamma^{-1} \tau \gamma \sigma^{-1} \gamma^{-1} \tau^{-1}  )  \\
	&\subseteq \supp(\gamma\sigma \gamma^{-1}  )\cup \supp( \tau \gamma \sigma^{-1} \gamma^{-1} \tau^{-1}  ) \\
	&= \gamma(\supp(\sigma)) \cup \tau \gamma( \supp(\sigma)) \\
	&\subseteq \gamma(A_0\cup \tau(A_0)\cup B_0 )\cup \tau(A)\cup \tau(C) \\
	&\subseteq  A_0 \cup C \cup \tau(A) \cup \tau(C).
	\end{align*}
	We have seen that $C\cup \tau(A) \cup \tau(C)$ is disjoint from $A$. Since $A_0$ is a proper subset of $A$, the support of $\tau_0$ does not contain $A\setminus A_0$ and so $\supp(\tau_0) \neq  \mathcal{G}^{(0)}$. It follows that
	\[
	\supp(\tau_1) = \gamma^{-1}(\supp(\tau_0)) \neq \mathcal{G}^{(0)}.
	\]
	Meanwhile, for any $x\in A_0$, we have
	\[
	\tau_1(x) = \sigma \gamma^{-1} (\tau \gamma \sigma^{-1} \gamma^{-1} \tau^{-1}) \gamma(x) =  \sigma(x) = \tau(x).   
	\]
	This implies that $\tau_2 := \tau_1^{-1}\tau\in N$ is supported outside $A_0$, hence $\tau = \tau_1\tau_2$ is a desired decomposition.
\end{proof}

\begin{prop} \label{prop:Normality}
	Let $\mathcal{G}$ be a minimal, second countable, \'etale groupoid whose unit space is a Cantor space. Suppose $\mathcal{G}$ has comparison. If $N$ is a subgroup of $[\mathcal{G}]$ normalized by $[\mathcal{G}]'$, then $N$ is normal in $[\mathcal{G}]$. 
\end{prop}
\begin{proof}
	Given $\tau \in N$ and $\alpha \in [\mathcal{G}]$, we need to show that $\alpha\tau\alpha^{-1} \in N$. In view of Lemma \ref{lem:DecompNormal} we may assume that there is a clopen set $B$ such that $\supp(\tau) \subseteq B \neq \mathcal{G}^{(0)}$. In addition, by Proposition \ref{prop:DecompSmall} and Lemma \ref{lem:Kerr}(1) we may assume that $\supp(\alpha)$ is contained in some clopen set $A$ satisfying $A \neq \mathcal{G}^{(0)}$ and $\mu(A) < \mu( \mathcal{G}^{(0)}\setminus B )$ for all $\mu\in M(\mathcal{G})$ (with $M(\mathcal{G})$ possibly being empty). Applying Lemma \ref{lem:ComparisonByFullGroup} we obtain an element $\gamma\in [[\mathcal{G}]]$ ($\subseteq [\mathcal{G}]$) such that $\gamma(A) \subseteq \mathcal{G}^{(0)}\setminus B$. In particular $\supp( \gamma \alpha \gamma^{-1} ) = \gamma( \supp(\alpha)  )$ is disjoint from $\supp(\tau)$, so $\gamma \alpha \gamma^{-1}$ and $\tau$ commute. It follows that
	\begin{align*}
	\alpha \tau \alpha^{-1} &= \alpha( \gamma \alpha^{-1} \gamma^{-1}  ) (\gamma \alpha \gamma^{-1}) \tau \alpha^{-1} = \alpha( \gamma \alpha^{-1} \gamma^{-1}  )\tau (\gamma \alpha \gamma^{-1})  \alpha^{-1} \\
	&= [\alpha, \gamma]\tau [ \alpha, \gamma ]^{-1},
	\end{align*}
	which belongs to $N$ because $N$ is normlized by $[\mathcal{G}]'$.
\end{proof}

\begin{lemma} \label{lem:CommutatorProduct}
	Let $G$ be a group and $N$ be a normal subgroup of $G$. If $g_1, g_2, ..., g_n, h_1, h_2, ..., h_m$ are elements in $G$ such that each $[g_i, h_j]$ belongs to $N$, then $[g_1\cdots g_n, h_1\cdots h_m]$ also belongs to $N$.
\end{lemma}
\begin{proof}
	Observe that
	\[
	[g_1g_2, h] = g_1g_2hg_2^{-1}g_1^{-1}h^{-1} = g_1[g_2, h]hg_1^{-1}h^{-1} = g_1[g_2, h]g_1^{-1}[g_1, h]
	\]
	and similarly
	\[
	[g, h_1h_2] = [g, h_1]h_1[g, h_2]h_1^{-1}.
	\]
	Since $N$ is normal, this proves the case $n= m = 2$. The general case follows by induction. 
\end{proof}

\begin{prop} [{cf.\ \cite[Theorem 4.7, Theorem 4.16]{Matui:Crelle}}] \label{prop:containment}
	Let $\mathcal{G}$ be a minimal, second countable, \'etale groupoid whose unit space is a Cantor space. Suppose $\mathcal{G}$ has comparison. Then any nontrivial normal subgroup of $[\mathcal{G}]$ contains $[\mathcal{G}]'$. 
\end{prop}
\begin{proof}
	Let $N$ be a nontrivial normal subgroup of $[\mathcal{G}]$. We first prove the following claim: given $\alpha, \beta\in [\mathcal{G}]$, if there is an element $\gamma\in N$ such that $\gamma(\supp(\alpha)) \cap \supp(\beta) = \emptyset$, then the commutator $[\alpha, \beta]$ belongs to $N$. Indeed, since 
	\[ \supp( \gamma \alpha^{-1} \gamma^{-1} ) = \gamma( \supp(\alpha^{-1})) = \gamma( \supp(\alpha)),
	\]
	the elements $\gamma\alpha^{-1}\gamma^{-1}$ and $\beta$ commute. Therefore
	\begin{align*}
	[\alpha, \beta] &= \alpha\beta \alpha^{-1}\beta^{-1} = \alpha \beta( \gamma \alpha^{-1} \gamma^{-1}  )( \gamma \alpha \gamma^{-1}) \alpha^{-1}\beta^{-1} \\
	&= \alpha ( \gamma \alpha^{-1} \gamma^{-1}  )\beta( \gamma \alpha \gamma^{-1}) \alpha^{-1}\beta^{-1} \\
	&= [\alpha, \gamma]\beta [\alpha, \gamma]^{-1} \beta^{-1}.
	\end{align*}
	As $N$ is normal in $[\mathcal{G}]$, the commutator $[\alpha, \gamma]$ is in $N$ and hence so is $[\alpha, \beta]$, which establishes the claim.

	Now we prove the theorem. Let $\alpha, \beta$ be two elements in $[\mathcal{G}]$. We need to show that the commutator $[\alpha, \beta]$ belongs to $N$. By Proposition \ref{prop:DecompSmall} and Lemma \ref{lem:CommutatorProduct} we may assume that there is a clopen set $B$ such that $\supp(\beta) \subseteq B \neq \mathcal{G}^{(0)}$. Let $\tau_0$ be a nontrivial element in $N$, and let $C$ be a clopen set such that $\tau_0(C) \cap C =\emptyset$. By Lemma \ref{lem:Kerr}(1) there is a constant $\eta > 0$ satisfying 
	\[
	\min\{  \mu(\mathcal{G}^{(0)}\setminus B),\; \mu(C)  \} \geq \eta 
	\]
	for all $\mu\in M(\mathcal{G})$ (if $M(\mathcal{G})$ is nonempty). Applying Proposition \ref{prop:DecompSmall} and Lemma \ref{lem:CommutatorProduct} to the element $\alpha$ we may assume that there is a clopen set $A$ such that
	\begin{itemize}
		\item $\supp(\alpha)\subseteq A \neq \mathcal{G}^{(0)}$, and
		\item $\mu(A) < \frac{\eta}{2}$ for all $\mu\in M(\mathcal{G})$.
	\end{itemize}
	Then by Lemma \ref{lem:ComparisonByFullGroup} there is an element $\gamma_0 \in [[\mathcal{G}]]$ ($\subseteq [\mathcal{G}]$) such that $\gamma_0(A) \subseteq \mathcal{G}^{(0)}\setminus B$ and $A\cup \supp(\gamma_0) \neq \mathcal{G}^{(0)}$. Moreover, by the same lemma, if $M(\mathcal{G}) \neq \emptyset$ then we may require that $\supp(\gamma_0)\subseteq A\cup \gamma_0(A)$. It follows that, in either case, we can find a clopen set $D$ containing $A\cup \supp(\gamma_0)$ and an element $\sigma\in [\mathcal{G}]$ (using Lemma \ref{lem:ComparisonByFullGroup} again) such that $\sigma(D) \subseteq C$.  
	
	We define $\tau := \sigma^{-1} \tau_0 \sigma$ and $\gamma := [ \gamma_0, \tau  ]$. Then $\tau( D ) \cap D = \emptyset$ and $\tau$ is in  $N$ because $N$ is normal in $[\mathcal{G}]$ by Proposition \ref{prop:Normality}. It follows that $\gamma$ also belongs to $N$. Since $\tau^{-1}(D)$ is disjoint from $D$ and $\supp(\gamma_0)$ is contained in $D$, the map $\gamma_0^{-1}$ acts trivially on $\tau^{-1}(D)$, and in particular on $\tau^{-1}(A)$. From this we deduce that
	\[
	\gamma(A) = \gamma_0\tau \gamma_0^{-1}\tau^{-1}(A) = \gamma_0(A) \subseteq \mathcal{G}^{(0)}\setminus B.
	\]
	In particular $\gamma( \supp(\alpha)  ) \cap \supp(\beta) =\emptyset$, which completes the proof thanks to the claim in the first paragraph of the proof.
\end{proof}

\begin{thm} \label{thm:Simplicity}
	Let $\mathcal{G}$ be a minimal, second countable, \'etale groupoid whose unit space is a Cantor space. Suppose $\mathcal{G}$ has comparison. Then the commutator subgroup $[\mathcal{G}]'$ of the full group $[\mathcal{G}]$ is simple. 
\end{thm}
\begin{proof}
	Suppose $N$ is a nontrivial normal subgroup of $[\mathcal{G}]'$. By Proposition \ref{prop:Normality} $N$ is normal in $[\mathcal{G}]$. Therefore by Proposition \ref{prop:containment} $N$ contains $[\mathcal{G}]'$, which implies that they are equal.
\end{proof}

It is worth mentioning that, according to Theorem \ref{thm:Simplicity}, the full group $[\mathcal{G}]$ itself is simple if and only if it is perfect, i.e, equal to its commutator subgroup. Whether $[\mathcal{G}]$ is perfect is an open question even for Cantor minimal systems (see \cite{IbarluciaMelleray:ETDS}). On the other hand, it is known that the topological full group $[[\mathcal{G}]]$ is not perfect (one can deduce this from the existence of the index map; see \cite{GPS:Israel} and \cite{Matui:PLMS}).

\bibliography{SimplicityCommutator_2}

\begin{thebibliography}{10}

\bibitem{ABBL:arXiv}
Pere Ara, Christian B{\"o}nicke, Joan Bosa, and Kang Li.
\newblock The type semigroup, comparison and almost finiteness for ample
  groupoids.
\newblock arXiv:2001.00376.

\bibitem{BM:Colloq}
Sergey Bezuglyi and Konstantin Medynets.
\newblock Full groups, flip conjugacy, and orbit equivalence of {C}antor
  minimal systems.
\newblock {\em Colloq. Math.}, 110(2):409--429, 2008.

\bibitem{DZ:arXiv}
Tomasz Downarowicz and Guohua Zhang.
\newblock The comparison property of amenable groups.
\newblock arXiv:1712.05129.

\bibitem{Dye:AJM1}
Henry Dye.
\newblock On groups of measure preserving transformations. {I}.
\newblock {\em Amer. J. Math.}, 81:119--159, 1959.

\bibitem{Dye:AJM2}
Henry Dye.
\newblock On groups of measure preserving transformations. {II}.
\newblock {\em Amer. J. Math.}, 85:551--576, 1963.

\bibitem{Eigen:Israel}
Stanley Eigen.
\newblock On the simplicity of the full group of ergodic transformations.
\newblock {\em Israel J. Math.}, 40(3-4):345--349 (1982), 1981.

\bibitem{Exel:PAMS}
Ruy Exel.
\newblock Reconstructing a totally disconnected groupoid from its ample
  semigroup.
\newblock {\em Proc. Amer. Math. Soc.}, 138(8):2991--3001, 2010.

\bibitem{Fathi:Israel}
Albert Fathi.
\newblock Le groupe des transformations de {$[0, 1]$} qui pr\'{e}servent la
  mesure de {L}ebesgue est un groupe simple.
\newblock {\em Israel J. Math.}, 29(2-3):302--308, 1978.

\bibitem{GPS:Israel}
Thierry Giordano, Ian~F. Putnam, and Christian~F. Skau.
\newblock Full groups of {C}antor minimal systems.
\newblock {\em Israel J. Math.}, 111:285--320, 1999.

\bibitem{GlasnerWeiss:IJM}
Eli Glasner and Benjamin Weiss.
\newblock Weak orbit equivalence of {C}antor minimal systems.
\newblock {\em Internat. J. Math.}, 6(4):559--579, 1995.

\bibitem{Hopf:TAMS}
Eberhard Hopf.
\newblock Theory of measure and invariant integrals.
\newblock {\em Trans. Amer. Math. Soc.}, 34(2):373--393, 1932.

\bibitem{HW:Book}
Witold Hurewicz and Henry Wallman.
\newblock {\em Dimension {T}heory}.
\newblock Princeton Mathematical Series, v. 4. Princeton University Press,
  Princeton, N. J., 1941.

\bibitem{IbarluciaMelleray:ETDS}
Tom\'{a}s Ibarluc\'{\i}a and Julien Melleray.
\newblock Full groups of minimal homeomorphisms and {B}aire category methods.
\newblock {\em Ergodic Theory Dynam. Systems}, 36(2):550--573, 2016.

\bibitem{JR:JFA}
Paul Jolissaint and Guyan Robertson.
\newblock Simple purely infinite {$C^\ast$}-algebras and {$n$}-filling actions.
\newblock {\em J. Funct. Anal.}, 175(1):197--213, 2000.

\bibitem{JM:Annals}
Kate Juschenko and Nicolas Monod.
\newblock Cantor systems, piecewise translations and simple amenable groups.
\newblock {\em Ann. of Math. (2)}, 178(2):775--787, 2013.

\bibitem{Kechris:Book}
Alexander~S. Kechris.
\newblock {\em Classical descriptive set theory}, volume 156 of {\em Graduate
  Texts in Mathematics}.
\newblock Springer-Verlag, New York, 1995.

\bibitem{Kerr:JEMS}
David Kerr.
\newblock Dimension, comparison, and almost finiteness.
\newblock {\em J. Eur. Math. Soc. (JEMS)}, 22(11):3697--3745, 2020.

\bibitem{KerrSzabo}
David Kerr and G\'{a}bor Szab\'{o}.
\newblock Almost finiteness and the small boundary property.
\newblock {\em Comm. Math. Phys.}, 374(1):1--31, 2020.

\bibitem{Krieger:MathAnn}
Wolfgang Krieger.
\newblock On a dimension for a class of homeomorphism groups.
\newblock {\em Math. Ann.}, 252(2):87--95, 1979/80.

\bibitem{Ma:arXiv}
Xin Ma.
\newblock Purely infinite locally compact hausdorff {\`e}tale groupoids and
  their {$C^*$}-algebras.
\newblock arXiv:2001.03706.

\bibitem{Matui:IJM}
Hiroki Matui.
\newblock Some remarks on topological full groups of {C}antor minimal systems.
\newblock {\em Internat. J. Math.}, 17(2):231--251, 2006.

\bibitem{Matui:PLMS}
Hiroki Matui.
\newblock Homology and topological full groups of \'{e}tale groupoids on
  totally disconnected spaces.
\newblock {\em Proc. Lond. Math. Soc. (3)}, 104(1):27--56, 2012.

\bibitem{Matui:Crelle}
Hiroki Matui.
\newblock Topological full groups of one-sided shifts of finite type.
\newblock {\em J. Reine Angew. Math.}, 705:35--84, 2015.

\bibitem{Renault:Book}
Jean Renault.
\newblock {\em A groupoid approach to {$C^{\ast} $}-algebras}, volume 793 of
  {\em Lecture Notes in Mathematics}.
\newblock Springer, Berlin, 1980.

\bibitem{Renault:Irish}
Jean Renault.
\newblock Cartan subalgebras in {$C^*$}-algebras.
\newblock {\em Irish Math. Soc. Bull.}, (61):29--63, 2008.

\bibitem{Sims:Notes}
Aidan Sims.
\newblock {\'E}tale groupoids and their {$C^*$}-algebras.
\newblock https://arxiv.org/abs/1710.10897.

\bibitem{Suzuki:IMRN}
Yuhei Suzuki.
\newblock {Almost Finiteness for General Étale Groupoids and Its Applications
  to Stable Rank of Crossed Products}.
\newblock {\em International Mathematics Research Notices}, 08 2018.
\newblock rny187.

\end{thebibliography}
\bibliographystyle{plain}

\end{document}